\documentclass[a4paper,10pt]{article}

\usepackage{newcent}
\usepackage{amsthm}
\usepackage{amsmath}
\usepackage{amssymb}

\newtheorem{theorem}{Theorem}[section]

\newtheorem{lemma}[theorem]{Lemma}
\newtheorem{corollary}[theorem]{Corollary}
\newtheorem{remark}[theorem]{Remark}

\newtheorem{assumption}[theorem]{Assumption}
\newtheorem{question}[theorem]{Question}

\title{On the identification of random variables\\ from quantized observations}

\author{Mikl\'os R\'asonyi\thanks{Alfr\'ed R\'enyi Institute of
Mathematics, Hungarian Academy of Sciences, Budapest.  
E-mail: \texttt{rasonyi@renyi.mta.hu}}}

\date{\today}

\begin{document}

\maketitle

\smallskip

\textsl{Dedicated to L\'aszl\'o Gerencs\'er on the occasion of
his 70th birthday.}

\smallskip

\begin{abstract} We prove that the scale and shift parameters of a family of probability laws can be
identified from quantized values, under appropriate assumptions. As an application,
we show the consistency of the maximum likelihood estimator for the parameters of a quantized Gaussian
autoregressive process.
\end{abstract}


\smallskip


\section{Description of the problem}

Observations of physical systems are constrained by the accuracy
of the devices performing the measurements. The system state can be 
determined with a certain precision only while system parameters may need to be estimated at a higher level of precision,
based on multiple observations of a (stochastic) process. 
The first step towards solving such a statistical problem is to prove that the system is identifiable, i.e. different 
parameter values correspond to different laws for the quantized process.

In this section we consider quantized observations of a real-valued random variable $X$. 
It is convenient to use the integer part function 
$$
h(x):=k \mbox{ for }x\in [k,k+1),\ k\in\mathbb{Z}.
$$

We study the two-parameter family of probability laws $\mathrm{Law}(h(\sigma X+\mu))$, $\sigma>0$,
$\mu\in\mathbb{R}$
and wish to know whether $\sigma$, $\mu$ are identifiable.
%We seek answer for the following question: does the law of $h(X+\mu)$ identify $\mu$ ? 
%We formulate this in a more precise way.

\begin{question}\label{lv}
Does $\mathrm{Law}(h(\sigma_1 X+\mu_1))=
\mathrm{Law}(h(\sigma_2 X+\mu_2))$ imply $\sigma_1=\sigma_2$ and $\mu_1=\mu_2$ ?
\end{question}

This question was answered in the affirmative by \cite{gmsz} for Gaussian random variables. In that paper a multidimensional
$X$ and a quantizer function of a very general type was considered. It was shown that the mean and the
convariance matrix of $X$ are identifiable from quantized observations. Unfortunately, \cite{gmsz} is unpublished.
This motivated us to search for results pertaining to non-Gaussian random variables as well, in the one-dimensional 
case and for the particular quantizer function $h$.

In Section \ref{ketto} we will provide an affirmative answer to Question \ref{lv} under Assumptions
\ref{f} and \ref{be} below. Relying on results of \cite{r}, in Section \ref{alk} we present a concrete 
statistical application: the estimation of the  
parameters for a quantized Gaussian autoregressive process.

\section{Identification of the shift and scale parameters}\label{ketto}

We will be working under the following assumptions. 

\begin{assumption}\label{f}
The cumulative distribution function $F(x):=P(X<x)$, $x\in\mathbb{R}$ is continuous and it is
strictly increasing.
\end{assumption}

\begin{assumption}\label{be}
For every $\delta>1$, one has either
\begin{equation}\label{laudes}
\lim_{x\to\infty}\frac{1-F(\delta x)}{1-F(x)}=0,
\end{equation}
or
\begin{equation}\label{organi}
\lim_{x\to-\infty}\frac{F(\delta x)}{F(x)}=0.
\end{equation}
\end{assumption}

\begin{remark}\label{gaussian}{\rm Let $X$ be standard Gaussian. Then Assumption \ref{f} holds trivially. Recall the 
folklore inequality
$$
\frac{e^{-x^2/2}}{\sqrt{2\pi}}\left(\frac{1}{x}-\frac{1}{x^3}\right)\leq 1-F(x)\leq
\frac{e^{-x^2/2}}{\sqrt{2\pi}}\frac{1}{x},\quad x>0,
$$
which shows that \eqref{laudes} holds for $X$. Then it is easy to
conclude that Assumptions \ref{f}, \ref{be} hold for any Gaussian random variable $X$.}
\end{remark}

Our first result is the following.

\begin{lemma}\label{mainx} Let $\mu_1,\mu_2\in\mathbb{R}$ and $\sigma_1,\sigma_2>0$. Under Assumption \ref{be}, 
$$\mathrm{Law}(h(\sigma_1 X+\mu_1))=\mathrm{Law}(h(\sigma_2 X+\mu_2))$$ implies
$\sigma_1=\sigma_2$.
\end{lemma}
\begin{proof}
Define the quantities
\begin{equation}\label{p}
p^i_k:=P(\sigma_i X+\mu_i\in [k, k+1)),\quad k\in\mathbb{Z},\ i=1,2,
\end{equation}
which describe the laws of $h(\sigma_i X+\mu_i)$, $i=1,2$.
We will show that if $\sigma_1<\sigma_2$ then these two
sequences of probabilities necessarily differ.
We work assuming 
\eqref{laudes}, the arguments in the case of \eqref{organi} being
analogous. Set $r^i_n:=\sum_{j=n}^{\infty}p^i_j$, $n\in\mathbb{Z}$.
For each $0<\eta<1$ 
$$
\frac{r^1_n}{r^2_n}\leq \frac{P(\sigma_1 X+|\mu_1|\geq n)}{P(\sigma_2 X - |\mu_2|\geq n)}\leq \frac{P(\sigma_1 X\geq \eta n)}{P(\sigma_2 X\geq n/\eta)},
$$
if $n$ is large enough. The latter quantity can be written as
$$
\frac{P\left( X\geq \frac{\eta n}{\sigma_1}\right)}{P\left( X\geq \frac{n}{\eta\sigma_2}\right)}.
$$
Choosing $\eta$ close enough to $1$, ${\eta}/{\sigma_1}>1/(\sigma_2\eta)$ holds. It follows then from \eqref{laudes}
that 
$$
\frac{r^1_n}{r^2_n}\to 0,\ n\to\infty,
$$
finishing the proof.
\end{proof}

\begin{lemma}\label{manx} 
Let $\mu_1,\mu_2\in\mathbb{R}$. Under Assumption \ref{f}, 
$$\mathrm{Law}(h(X+\mu_1))=\mathrm{Law}(h(X+\mu_2))$$ implies
$\mu_1=\mu_2$. 
\end{lemma}
\begin{proof}
This proof was suggested by an anonymous reviewer of an earlier version of the paper. We may and will assume
$\mu_1=0$ and $\mu_2\geq 0$. Assumption \ref{f} and $p^1_k=p^2_k$, $k\in\mathbb{Z}$ imply
that, for all $k\in\mathbb{Z}$,
$$
F(k+1+\mu_2)-F(k+\mu_2)=F(k+1)-F(k),
$$
which gives 
\begin{equation}\label{pap}
F(k+1+\mu_2)-F(k+1)=F(k+\mu_2)-F(k)=:C,\ k\in\mathbb{Z},
\end{equation}
with a constant $C$ independent of $k$.
Denoting by $N$ any integer with $\mu_2\leq N$, we get that
$$
\sum_{i=0}^{\infty} C\leq\sum_{i=0}^{\infty}[F(k+(i+1)N)-F(k+iN)]\leq 1,
$$
so $C=0$ necessarily. If $\mu_2>1$ then \eqref{pap} implies $F(k+1)=F(k)$ for
all $k$, which is clearly impossible. If $0<\mu_2\leq 1$ then, by \eqref{pap} again, $F$
can only increase on intervals of the form $[k+\mu_2,k+1]$, $k\in\mathbb{Z}$ but this
contradicts Assumption \ref{f}. So necessarily $\mu_2=0$.
\end{proof}

\begin{theorem}\label{main} Let $\mu_1,\mu_2\in\mathbb{R}$ and $\sigma_1,\sigma_2>0$. Under Assumptions \ref{f} and \ref{be}, 
$\mathrm{Law}(h(\sigma_1 X+\mu_1))=\mathrm{Law}(h(\sigma_2 X+\mu_2))$ implies
$\sigma_1=\sigma_2$ and $\mu_1=\mu_2$.
\end{theorem}
\begin{proof}[Proof of Theorem \ref{main}.] Lemma \ref{mainx} shows $\sigma_1=\sigma_2$. Notice that Assumption
\ref{f} holds for the random variable $\sigma_1 X$, too, hence Lemma \ref{manx} completes the proof.
\end{proof}

In Section \ref{alk} we will use the ``rounding-off'' function $q$, defined by
\[
q(x)=k,\ \mbox{ for }x\in [k-1/2,k+1/2),\ k\in\mathbb{Z}.
\]

\begin{corollary}\label{maine} Under Assumptions \ref{f} and \ref{be}, 
$$
\mathrm{Law}(q(\sigma_1 X+\mu_1))=\mathrm{Law}(q(\sigma_2 X+\mu_2))
$$ implies
$\sigma_1=\sigma_2$ and $\mu_1=\mu_2$.
\end{corollary}
\begin{proof}
Note that $\mathrm{Law}(h(Y))=\mathrm{Law}(q(Y-1/2))$, for all
random variables $Y$. Thus the statement follows from Theorem \ref{main}.
\end{proof}

\section{Quantized autoregressive processes}\label{alk}

Here we consider the problem of estimating  $\theta$ from a sample of quantized (``rounded off'') observations 
$(q(X_0^{(\theta)}),\ldots,q(X_n^{(\theta)}))$,
where $X_n^{(\theta)}$, $n\in\mathbb{N}$ are a family of stochastic processes parametrized by $\theta$. 
The case of i.i.d. $X_n^{(\theta)}$
can be attacked by standard methods. However, if $X_n^{(\theta)}$ is e.g. Markovian then we face a much more delicate problem.

The versatility of Gaussian ARMA processes makes them a natural class for investigating the effect of rounding off. 
Closely related questions have been addressed in
\cite{torma}, where an algorithm for the calculation of related maximum likelihood (ML) estimates has been suggested, see 
also \cite{r,qsajat}. There has been,
however, no theoretical justification for the convergence of such ML estimates
in the literature so far. 

In the present paper we establish the consistency
of the ML estimate of the parameters of a stable Gaussian AR(1)
process from rounded off observations. 
Due to the severe discontinuity of the function $q$ 
this is a hard problem of mathematical statistics and our analysis
relies on results of \cite{r} whose proofs involved substantial
technicalities. 

We are working on a fixed probability space $(\Omega,\mathcal{F},P)$ throughout this section.
Let the stable Gaussian AR(1) process $X_n$ be given by 
$$
X_n=\alpha_* X_{n-1}+\varepsilon_n,\quad n\geq 0,
$$
where $\varepsilon_n$ are i.i.d. random variables with
law $N(\mu_*,\sigma_*^2)$ and $X_{-1}\in\mathbb{R}$ is a deterministic
initial value. For simplicity, we assume $X_{-1}=0$.
We also assume that $\alpha_*$, $\mu_*$, $\sigma_*$ are unknown but they are supposed to lie in 
intervals of admissible parameters $[\underline{\alpha},\overline{\alpha}]\subset [0,1)$,
$[\underline{\mu},\overline{\mu}]\subset\mathbb{R}$,
$[\underline{\sigma},\overline{\sigma}]\subset (0,\infty)$, respectively. Set
$$
\Theta:=[\underline{\alpha},\overline{\alpha}]\times
[\underline{\mu},\overline{\mu}]\times [\underline{\sigma},\overline{\sigma}].
$$

Only the rounded-off values $Y_n:=q(X_n)$, $n\geq 0$ are observed.
Using Corollary \ref{maine} above and Lemma \ref{arx} below, we will prove that the maximum likelihood (ML) 
estimates $\tilde{\theta}_n$ 
for $\theta_*:=(\alpha_*,\mu_*,\sigma_*)\in\Theta$, based on the observation sequence $Y_n$, $n\geq 0$,  
converge almost surely to $\theta_*$ (see below for the precise definitions).  

For each $\theta\in\Theta$, let us define
\[
X_{n}^{(\theta)}=\alpha X_{n-1}^{(\theta)}+\varepsilon_{n}^{(\theta)},\ n\geq 0,
\]
setting $X_{-1}^{(\theta)}:=0$. Here $\varepsilon_n^{(\theta)}$, $n\geq 0$ are
i.i.d. with law $N(\mu,\sigma^2)$. The process $X_n^{(\theta)}$
describes the evolution of the system under the hypothesis that
the parameter is $\theta$. We clearly have 
$$
\mathrm{Law}(X_n^{(\theta_*)},\,n\in\mathbb{N})=\mathrm{Law}(X_n,\, n\in\mathbb{N}),
$$
which provides the true system dynamics. Set also
$Y_n^{(\theta)}:=q(X_n^{(\theta)})$.

The loglikelihood functions are defined as
\begin{eqnarray*}
L_n(\theta;y_0,\ldots,y_n) &:=& \ln P(Y_0^{(\theta)}=y_0,\ldots,
Y_n^{(\theta)}=y_n),
\end{eqnarray*}
for $y_0,\ldots,y_n\in\mathbb{Z}$ and $\theta\in\Theta$. The ML estimates we are dealing
with are defined as a sequence of random variables $\tilde{\theta}_n$
such that
\[
\tilde{\theta}_n\in\mathrm{argmax}_{\theta}L_n(\theta;Y_0,\ldots,Y_n).
\]
We will also
need
\begin{eqnarray*}
S_0(\theta;y_0) &:=& P(Y_0^{(\theta)}=y_0),\\ 
S_n(\theta;y_0,\ldots,y_n) &:=&
P(Y_n^{(\theta)}=y_n\vert Y_0^{(\theta)}=y_0,\ldots,Y_{n-1}^{(\theta)}=y_{n-1}),\ n\geq 1,
\end{eqnarray*}
and $G_n(\theta):=ES_n(\theta;Y_0,\ldots,Y_n)$. Define 
$$
F_n(\theta):=\frac{1}{n}EL_n(\theta;Y_0,\ldots,Y_n)=\frac{1}{n}\sum_{j=0}^n G_j(\theta).
$$

We recall some of the conclusions of earlier work on the subject.
Theorem 1.1 and Remark 5.11
of \cite{r} imply that there is a continuous function $\mathbf{L}:\Theta\to\mathbb{R}$ 
such that $(1/n)L_n(\theta;Y_0(\omega),\ldots,Y_n(\omega))$ converges
to $\mathbf{L}(\theta)$ uniformly in $\theta\in\Theta$, for almost all $\omega\in\Omega$, and also 
$F_n(\theta)$, $G_n(\theta)$ tend to $\mathbf{L}(\theta)$
as $n\to\infty$.
By Remark 5.11 of
\cite{r}, we actually have $\vert G_n(\theta)-\mathbf{L}(\theta)\vert\leq C\rho^n$
for some $C>0$ and $0<\rho<1$.

The main result of this section is the following.

\begin{theorem}\label{main2}
The sequence of estimators $\tilde{\theta}_n$, $n\in\mathbb{N}$ tends to $\theta_*$ almost 
surely as $n\to\infty$.
\end{theorem}

\begin{remark}
{\rm The actual calculation of $\tilde{\theta}_n$ is rather difficult. An algorithm based on Markov chain Monte Carlo methods was analysed in \cite{torma},
another one based on particle filters is in \cite{qsajat}.}
\end{remark}
 
Now a lemma in the spirit of Lemma \ref{mainx} above is presented.

\begin{lemma}\label{arx}
Let $Z$, $\epsilon$ be independent random variables such that $\epsilon$ satisfies Assumption \ref{be}. Let $\alpha_1,\alpha_2\geq 0$,
$\sigma_1,\sigma_2>0$ and
$\mu_1,\mu_2\in\mathbb{R}$. If
$$
\mathrm{Law}(q(Z),q(\alpha_1 Z+\epsilon))=\mathrm{Law}(q(Z),q(\alpha_2 Z+(\sigma_2/\sigma_1)(\epsilon-\mu_1)+\mu_2))
$$
then $\sigma_1=\sigma_2$.
\end{lemma}
\begin{proof}
Fix $k$ such that $P(Z\in [k-1/2,k+1/2))>0$. We explain only the case where the cumulative distribution function of
$\epsilon$ satisfies \eqref{laudes}, the other case being analogous. For any $\eta<1$ we have, for $n\in\mathbb{N}$ large enough,
\begin{eqnarray}\nonumber
\frac{P(Z\in [k-1/2,k+1/2),\ \alpha_1 Z+\epsilon\geq n-1/2)}{P(Z\in [k-1/2,k+1/2),\ \alpha_2 Z+
(\sigma_2/\sigma_1)(\epsilon-\mu_1)+\mu_2\geq n-1/2)} &\leq&\\
\nonumber \frac{P(Z\in [k-1/2,k+1/2),\ \epsilon\geq \eta n)}{P(Z\in [k-1/2,k+1/2),\ 
(\sigma_2/\sigma_1)\epsilon \geq n/\eta)} &=&\\
\label{cris}\frac{P(\epsilon\geq \eta n)}{P(\epsilon \geq (\sigma_1 n)/(\sigma_2\eta))}, & &
\end{eqnarray}
using independence in the last equality. By Assumption \ref{be}, the quantity \eqref{cris} tends to $0$ as $n\to\infty$  
if $\sigma_1<\sigma_2$ and $\eta$ is close enough to $1$.
\end{proof}

The proof of the next result uses ideas which are standard, we learnt them 
from Chapter 7 of \cite{rvh}. However, their adaptation to the present setting is non-trivial. In the sequel the notation $1_A$ refers to the indicator function of a set $A$.

\begin{lemma}\label{bht} We have
$\mathbf{L}(\theta)<\mathbf{L}(\theta_*)$ for all
$\theta\in\Theta$, $\theta\neq\theta_*$.
\end{lemma}
\begin{proof}
Denote by $P_n({\theta})$ the law of $(Y_0^{(\theta)},\ldots,Y_n^{(\theta)})$ on $\mathbb{Z}^{n+1}$, for $n\in\mathbb{N}$ and
by $P_{\infty}({\theta})$ the law of $(Y_0^{(\theta)},Y_1^{(\theta)},\ldots)$
on $\mathbb{Z}^{\mathbb{N}}$. Expectation under $P_{\infty}({\theta})$ will be denoted by $E^{\theta}$.
Define $W_n:=dP_{n}(\theta_*)/dP_{n}(\theta)$, $n\geq 1$ and notice that
$$
F_n(\theta_*)-F_n(\theta)=\frac{1}{n}E^{\theta} [W_n\ln W_n]
$$ equals $1/n$ times the Kullback-Leibler
divergence from $P_n({\theta})$ to $P_n({\theta_*})$  and is thus non-negative. 
Passing to the limit we get $\mathbf{L}(\theta)\leq \mathbf{L}(\theta_*)$,
for all $\theta\in\Theta$.

We claim that if $\mathbf{L}(\theta)=\mathbf{L}(\theta^*)$ then 
$P_{\infty}(\theta_*)\ll P_{\infty}(\theta)$. Indeed,
it is enough to show that the $W_n$, $n\geq 1$
are uniformly integrable. 

To this end, note that $\vert x\ln x- x\ln^+x\vert\leq 1/e$ (here $\ln^+$ denotes positive part of $\ln$)
and thus
\begin{eqnarray*}
E^{\theta}W_n\ln^+ W_n\leq (1/e)+E^{\theta}W_n\ln W_n= (1/e)+ n[F_n(\theta_*)-F_n(\theta)] &=&\\
(1/e)+ \sum_{j=0}^n [G_n(\theta_*)-G_n(\theta)] &=&\\
(1/e)+ \sum_{j=0}^n [G_n(\theta_*)-\mathbf{L}(\theta_*)-G_n(\theta)
+\mathbf{L}(\theta)] &\leq&\\
 (1/e) + \sum_{j=0}^{\infty} 2C\rho^j &=:& D<\infty
\end{eqnarray*}
for some constant $D$, independent of $n$. Uniform integrability of $W_n$ hence follows
by the de la Vall\'e-Poussin criterion.

We now prove that $P_{\infty}(\theta)\perp P_{\infty}(\theta_*)$ for
$\theta\neq\theta_*$ which will show, together with the previous arguments,
that $\mathbf{L}(\theta)=\mathbf{L}(\theta_*)$ is impossible for
$\theta\neq\theta_*$. 

Let $M(\theta)$ be a random variable with law $N\left(\frac{\mu}{1-\alpha},\frac{\sigma^2}{1-\alpha^2}\right)$
(notice that $X_n^{(\theta)}$ converges to $M(\theta)$ in law as $n\to\infty$).
Suppose first that either of 
\begin{equation}\label{tutto}
\frac{\mu}{1-\alpha}=\frac{\mu_*}{1-\alpha^*}\quad\mbox{ or }\quad\frac{\sigma^2}{1-\alpha^2}=\frac{\sigma_*^2}{1-\alpha_*^2}
\end{equation}
fail. Apply Corollary \ref{maine} with the choice $\mathrm{Law}(X)=N(0,1)$. Assumptions \ref{f} and \ref{be}
hold in this case, see Remark \ref{gaussian}. We get that
$$
\mathrm{Law}(q(M(\theta)))\neq \mathrm{Law}(q(M(\theta_*))).
$$

In particular, there is
$k\in\mathbb{Z}$ such that 
$$
A:=P(M(\theta)\in [k-1/2,k+1/2))\neq P(M(\theta_*)\in [k-1/2,k+1/2))=:B.
$$

Define the process $H_n(\theta):=1_{\{Y_n^{(\theta)}=k\}}$. 
As the law of $X_n^{(\theta)}$ tends to that of $M(\theta)$ weakly
and the latter has a continuous cumulative distribution function, 
it follows that 
\begin{equation}\label{most}
EH_n(\theta)=P(Y_n^{(\theta)}=k)\to A,\quad
EH_n(\theta_*)=P(Y_n=k)\to B,
\end{equation} 
as $n\to\infty$.

By Proposition 2.9 of \cite{r}, $Y_n^{(\theta)}$ is L-mixing (see \cite{gerencser}
or Definition 2.2. of \cite{r} for details on this concept) and
then $H_n(\theta)$ is easily seen to be L-mixing, too,
for all $\theta$.
(Indeed, take a Lipschitz-continuous function $\ell$ such that $\ell(k)=1$
and $\ell(x)=0$ for $x\notin (k-1/2,k+1/2)$; note that 
$H_n(\theta)=\ell(Y_n^{(\theta)})$ and recall Proposition 2.4 of \cite{r} which implies
that class of L-mixing processes is stable under Lipschitz mappings.)

Theorem 2.5 of \cite{r} (which is a direct consequence of Corollary 1.3 in \cite{gerencser}) and \eqref{most} imply that both
\[
\frac{1}{n}\sum_{j=0}^n H_j(\theta)\to A,\ \frac{1}{n}\sum_{j=0}^n H_j(\theta_*)\to B,
\ n\to\infty
\]
almost surely hold. But this means that $P_{\infty}({\theta_*})$
assigns probability $1$ to an event which has $P_{\infty}({\theta})$-probability
$0$, hence indeed $P_{\infty}({\theta})\perp P_{\infty}({\theta_*})$.

Now we are left with the case where both equalities of \eqref{tutto} hold, in particular, $\mathrm{Law}(M(\theta))=\mathrm{Law}(M(\theta^*))$.
Let $D(\theta)$ be a two-dimensional Gaussian random variable with mean and covariance given by
$$
\left(\begin{array}{c} \frac{\mu}{1-\alpha} \\ \frac{\mu}{1-\alpha} \end{array}\right)\mbox{ and }\left(
\begin{array}{cc} \frac{\sigma^2}{1-\alpha^2} & \frac{\alpha\sigma^2}{1-\alpha^2} \\ \frac{\alpha\sigma^2}{1-\alpha^2} & \frac{\sigma^2}{1-\alpha^2} \end{array}\right).
$$
Clearly, $D(\theta)$ can be written as $(Z(\theta),\alpha Z(\theta)+\epsilon(\theta))$ where
$Z(\theta)$ has the same law as $M(\theta)$ and $\epsilon(\theta)$ is independent of it
with the same law as $\varepsilon^{(\theta)}_0$. Note that the law of $D(\theta)$ is the
weak limit of that of $(X^{(\theta)}_n,X^{(\theta)}_{n+1})$ as $n\to\infty$.

In this case Lemma \ref{arx} allows to prove, using L-mixing and arguing as above, that $P_{\infty}(\theta)\perp P_{\infty}(\theta^*)$
unless $\sigma=\sigma_*$. Note that the latter property, together with \eqref{tutto}, imply also $\mu=\mu_*$ and $\alpha=\alpha_*$,
so the proof is complete.
\end{proof}

Now we are able to
show the consistency of the estimators
$\tilde{\theta}_n$.

\medskip

\noindent\textsl{Proof of Theorem \ref{main2}.}
By Lemma \ref{bht} above, $\mathbf{L}$ has its
unique maximum at $\theta_*$. Hence almost sure convergence
of $\tilde{\theta}_n$ to $\theta_*$
follows from Corollary 1.4 of \cite{r}. \hfill $\Box$

\medskip

\noindent\textbf{Acknowledgments.} {Supported by the
``Lend\"ulet'' Grant LP2015-6 of the Hungarian Academy of Sciences.
The author wishes to thank an anonymous referee for communicating an elegant and
simple proof of Lemma \ref{manx} which replaces complicated 
Fourier-analytic arguments of the previous version and works under weaker
hypotheses.}

\end{document}